\documentclass[11pt]{article}
\usepackage{makeidx}
\usepackage{amsfonts}
\usepackage{fullpage}
\usepackage{amsmath}
\usepackage{amssymb}
\usepackage{amsthm}
\usepackage{appendix}
\usepackage{manfnt}
\usepackage{marginnote}
\usepackage{xfrac}
\newtheorem{thm}{Theorem}
\newtheorem{lem}{Lemma}

\newtheorem{cor}{Corollary}
\theoremstyle{definition}

\newtheorem{defn}{Definition}

\theoremstyle{remark}
\newtheorem{rmk}{Remark}
\usepackage{theoremref}
\usepackage[all]{xy}
\usepackage{mathrsfs}
\newcommand{\op}{\mathrm{op}}
\newcommand{\fp}{\operatorname{fp}}

\newcommand{\Ab}{\mathrm{Ab}}

\newcommand{\Ext}{\mathrm{Ext}}

\newcommand{\lmods}{\text{-}\mathrm{mod}}
\newcommand{\Rmods}{\mathrm{Mod}\text{-}}
\newcommand{\Lmods}{\text{-}\mathrm{Mod}}

\newcommand{\Hom}{\mathrm{Hom}}

\newcommand{\Flat}{\mathrm{Flat}}
\newcommand{\Fpinj}{\mathrm{Fpinj}}
\newcommand{\Inj}{\mathrm{Inj}}
\newcommand{\C}{\mathcal{C}}
\newcommand{\A}{\mathcal{A}}

\begin{document}
\title{Characterisations of purity in a locally finitely presented additive category: A short functorial proof}
\author{Samuel Dean}

\maketitle
\abstract{In this expository article, we will give an efficient functorial proof of the equivalence of various characterisations of purity in a finitely accessible additive category $\C$. The complications of the proofs for specific choices of $\C$ are contained in the description of fp-injective and injective objects in $(\fp\C,\Ab)$, the category  of additive functors $\fp\C\to\Ab$. For example, the equivalence of many characterisations of purity in a module category $\A\Lmods$ is a simple corollary of what we will prove here, since we know which objects are fp-injective, and which objects are injective, in $(\A\lmods,\Ab)$.}
\section{Extending functors over direct limits}
\paragraph{Acknowledgements}I thank Sergio Estrada and Pedro Guil Asensio for inviting me to visit the University of Murcia, where I wrote this, and for helpful discussions. Nothing herein deserves to be called original, but it is meant to be helpful.

\emph{All categories and functors mentioned in this paper are additive.} We assume some background on locally finitely presented categories, which can be gotten from \cite{crawleyboevey1994}. We write $\Ab$ for the category of abelian groups and, for a small category $\A$, we write $(\A,\Ab)$ for the category of functors $\A\to\Ab$, and $\Flat(\A,\Ab)$ for the category of flat functors $\A\to\Ab$.  I will write \textbf{direct limit} to mean the same as directed colimit.
\begin{defn}Let $\C$ be a category with direct limits. An object $A\in\C$ is \textbf{finitely presented} if the representable functor 
\begin{displaymath}
\C(A,-):\C\to\Ab
\end{displaymath}
preserves direct limits. We write $\fp \C$ for the full subcategory of finitely presented objects in $C$
\end{defn}
\begin{defn}Let $\C$ be a category with direct limits. We say that $\C$ is \textbf{locally finitely presented} if $\fp \C$ is skeletally small and every object is a direct limit of finitely presented objects.
\end{defn}
\begin{thm}[{\cite{crawleyboevey1994}}]\thlabel{flatemb} For any locally finitely presented category $\C$, the functor 
\begin{displaymath}
\C\to ((\fp \C)^\op,\Ab):C\mapsto \C(-,C)|_{(\fp \C)^\op}
\end{displaymath}
is fully faithful and restricts to an equivalence $\C\simeq\Flat((\fp \C)^\op,\Ab)$.
\end{thm}
We will use tensor products of functors. For this, there are references such as \cite{fisher1968} and \cite{fisherpalmquistnewell1971}, but we offer the following definition.
\begin{defn}Let $A$ be a small category. For functors $G:\A^\op\to\Ab$ and $F:\A\to\Ab$, the tensor product $G\otimes_\A F$ is an abelian group given by the coend formula (see \cite{maclane1998} for coends) 
\begin{displaymath}
G\otimes_\A F=\int^{A\in \A}(GA)\otimes_\mathbb{Z}(FA).
\end{displaymath}
\end{defn}
\begin{lem}\thlabel{tensyon}Let $\A$ be a small category. For any object $A\in \A$ and any functor $F:\A\to\Ab$, there is an isomorphism
\begin{displaymath}
\A(-,A)\otimes_\A F\cong FA
\end{displaymath}
which is natural in $F$ and $A$.
\end{lem}
\begin{proof}
See \cite[Proposition 1.1]{fisher1968} or take this as an exercise in the calculus of coends.
\end{proof}
\begin{defn}Let $\C$ be a locally finitely presented category. For any functor $F:\fp \C\to\Ab$, define $\overrightarrow{F}:\C\to\Ab$ by
\begin{displaymath}
\overrightarrow{F}C=\C(-,C)|_{(\fp \C)^\op}\otimes _{\fp \C}F
\end{displaymath}
for any $C\in \C$.
\end{defn}
\begin{thm}Let $\C$ be a locally finitely presented category. For any functor $F:\fp \C\to\Ab$, $\overrightarrow{F}$ preserves direct limits and there is an isomorphism $\overrightarrow{F}|_{\fp \C}\cong F$ which is natural in $F$. If $E:\C\to\Ab$ preserves direct limits and $E|_{\fp \C}\cong F$ then $E\cong\overrightarrow{F}$.
\end{thm}
\begin{proof}
Variations of this statement appear in many places, but we will give a proof, for the sake of self-containment, which similar to that at \cite[3.16]{dean2017}. See \cite{auslander1976} for a very simple argument when $F$ is finitely presented.  

The property that $\overrightarrow{F}$ preserves direct limits and restricts to $F$ on $\fp \C$ follows directly from the definition of $\overrightarrow{F}$ and \thref{tensyon}.

For such a functor $E:\C\to\Ab$, let $\alpha:F\to E|_{\fp \C}$ be an in isomorphism and, for each $C\in \C$, assemble the morphisms
\begin{displaymath}
C(A,C)\otimes_{\mathbb{Z}}FA\to EC:f\otimes x\mapsto ((Ef)\alpha_A)x \hspace{10mm} (A\in\fp \C)
\end{displaymath}
each of which is natural in $C$, into a morphism
\begin{displaymath}
\overrightarrow{F}C=\C(-,C)|_{(\fp \C)^\op}\otimes F\to EC
\end{displaymath}
which is natural in $C$. This morphism is an isomorphism when $C\in\fp \C$. Since both $\overrightarrow{F}$ and $E$ preserve direct limits, it follows that this morphism is an isomorphism for any $C\in\C$.
\end{proof}
\section{Purity in a locally finitely presented category}
\begin{defn}Let $\C$ be a locally finitely presented category. A sequence 
\begin{displaymath}
0\to A\to B\to C\to 0
\end{displaymath}
in $\C$ is \textbf{pure-exact} if and only if the induced sequence
\begin{displaymath}
0\to \C(-,A)|_{(\fp \C)^\op}\to \C(-,B)|_{(\fp \C)^\op}\to \C(-,C)|_{(\fp \C)^\op}\to 0
\end{displaymath}
is exact.
\end{defn}
\begin{defn}For a functor $F:\A\to\Ab$, we define its \textbf{dual} to be the functor $F^*:\A^\op\to\Ab$ defined by $F^*A=\Hom_{\mathbb{Z}}(FA,\mathbb{Q}/\mathbb{Z})$.
\end{defn}
If the reader is working in a slightly different context, with a $k$-linear locally finitely presented category, and prefers to replace $\mathbb{Z}$ by $k$ and $\mathbb{Q}/\mathbb{Z}$ by some injective cogenerator in $k\Lmods$, then they may do so. The following theorem will still hold.
\begin{defn}A functor $F:\A\to\Ab$ is said to be \textbf{fp-injective} if $\Ext^1(-,F)|_{(\fp(\A,\Ab))^\op}=0$. We write $\Fpinj(\A,\Ab)$ for the category of all fp-injective functors $\A\to\Ab$ and $\Inj(\A,\Ab)$ for the category of all injective functors $\A\to\Ab$.
\end{defn}

\begin{rmk}Let $\C$ be a locally finitely presented category. It is equivalent to $\Flat((\fp\C)^\op,\Ab)$, which is closed under extensions in $((\fp\C)^\op,\Ab)$. By \cite[Lemma 10.20]{buhler}, this implies that any exact structure on $((\fp\C)^\op,\Ab)$ restricts to an exact structure on $\Flat((\fp\C)^\op,\Ab)$. The abelian exact structure $((\fp\C)^\op,\Ab)$ restricts to the exact structure on $\Flat((\fp\C)^\op,\Ab)$ which corresponds to the class of pure-exact sequences on $\C$. Therefore, the pure-exact sequences on $\C$ form an exact structure on $\C$. In particular, for any pure-exact sequence 
\begin{displaymath}
    0\to A\overset{f}{\to} B\overset{g}{\to} C\to 0
\end{displaymath}
in $\C$, the pullback of $g$ exists along any morphism to $C$, as does the pushout of $f$ along any morphism from $A$.
\end{rmk}

\begin{thm}Let $\C$ be a locally finitely presented category. For a sequence of maps
\begin{displaymath}
0\to A\to B\to C\to 0
\end{displaymath}
in $\C$, the following are equivalent.

\begin{enumerate}
\item It is pure-exact.
\item It is a direct limit of split exact sequences.
\item For any $F\in \fp(\fp\C,\Ab)$, the induced sequence
\begin{displaymath}
0\to \overrightarrow{F}A\to \overrightarrow{F}B\to \overrightarrow{F}C\to 0
\end{displaymath}
is exact in $\Ab$.
\item For any $F\in (\fp \C,\Ab)$, the induced sequence
\begin{displaymath}
0\to \overrightarrow{F}A\to \overrightarrow{F}B\to \overrightarrow{F}C\to 0
\end{displaymath}
is exact in $\Ab$.
\item For any $F\in \Fpinj(\fp \C,\Ab)$, the induced sequence
\begin{displaymath}
0\to \overrightarrow{F}A\to \overrightarrow{F}B\to \overrightarrow{F}C\to 0
\end{displaymath}
is exact in $\Ab$.
\item For any $F\in \Inj(\fp \C,\Ab)$, the induced sequence
\begin{displaymath}
0\to \overrightarrow{F}A\to \overrightarrow{F}B\to \overrightarrow{F}C\to 0
\end{displaymath}
is exact in $\Ab$.
\item The induced sequence
\begin{displaymath}
0\to \C(-,C)|_{(\fp \C)^\op}^*\to \C(-,B)|_{(\fp \C)^\op}^*\to \C(-,A)|_{(\fp \C)^\op}^*\to 0
\end{displaymath}
is split exact in $(\fp C,\Ab)$.
\end{enumerate}
\end{thm}
\begin{proof}
1 implies 2: This argument is well-known and standard, but we give it for the sake of self-containment. Express $C$ as a direct limit of finitely presented objects, $C=\underrightarrow{\lim}_{\lambda\in\Lambda}C_\lambda$. Since pure-exact sequences form an exact structure on $C$, the pullback of any pure epimorphism along any other morphism exists and is a pure epimorphism. Take the pullback of our sequence along the morphisms
$$C_\lambda\to C.$$
We obtain a directed system of pure-exact sequences$$0\to A\to B_\lambda\to C_\lambda\to 0,$$each of which must be split since $C_\lambda$ is finitely presented. The direct limit of this sequence is our original sequence.

2 implies 3: Obvious: Every functor preserves split exact sequences, and the direct limit of any split exact sequence is exact.

3 implies 4: For any object $D\in C$, the functor $(\fp \C,\Ab)\to \Ab:F\mapsto\overrightarrow{F}D$ clearly preserves direct limits because it is a tensor product. By expressing $F$ as a direct limit of finitely presented functors, $F=\underrightarrow{\lim}F_\lambda$, we obtain the sequence $$0\to \overrightarrow{F}A\to \overrightarrow{F}B\to\overrightarrow{F}C\to 0$$as a direct limit of pure-exact sequence$$0\to \overrightarrow{F_\lambda}A\to \overrightarrow{F_\lambda}B\to\overrightarrow{F_\lambda}C\to 0,$$which is exact since since direct limits are exact.

4 implies 5: Obvious.

5 implies 6: Obvious.

6 implies 7: To show that our sequence is split, we need only show that, for any $F\in\Inj(\fp \C,\Ab)$, the sequence 
\begin{displaymath}
0\to (F,\C(-,C)|^*_{(\fp \C)^\op})\to (F,\C(-,B)|^*_{(\fp \C)^\op})\to (F,\C(-,A)|^*_{(\fp \C)^\op})\to 0
\end{displaymath}
is exact. The reason for this is that, since it is the dual of a flat functor, $\C(-,A)|_{\fp \C}^*$ is injective (there is a standard argument for this -- see e.g. \cite[19.14]{andersonfuller} for something similar), and therefore we may substitute $F=\C(-,A)^*$ to obtain a splitting.

Indeed, if $F\in\Inj(\fp \C,\Ab)$ then, by the hom-tensor duality, this sequence is isomorphic to 
\begin{displaymath}
0\to (\C(-,C)|_{(\fp \C)^\op}\otimes_{\fp \C}F)^*\to (\C(-,B)|_{(\fp \C)^\op}\otimes_{\fp \C}F)^*\to (\C(-,A)|_{(\fp \C)^\op}\otimes_{\fp \C}F)^*\to 0,
\end{displaymath}
which is equal to
\begin{displaymath}
0\to \left(\overrightarrow{F}B\right)^*\to\left(\overrightarrow{F}B\right)^*\to\left(\overrightarrow{F}A\right)^*\to 0
\end{displaymath}
which is exact by hypothesis.

7 implies 1: Easy since $\mathbb{Q}/\mathbb{Z}$ is an injective cogenerator.
\end{proof}
\begin{cor}[Well-known]For a pre-additive category $\A$ and a sequence
\begin{displaymath}
0\to L\to M\to N\to 0
\end{displaymath}
in $\A\Lmods=(\A,\Ab)$, the following are equivalent.
\begin{enumerate}
\item It is pure-exact.
\item It is a direct limit of split exact sequences.
\item For any pp-pair $\varphi/\psi$ in the language of left $\A$-modules, the sequence
\begin{displaymath}
0\to \varphi L/\psi L\to\varphi M/\psi M\to\varphi N/\psi N\to 0
\end{displaymath}
is exact.
\item For any $Y\in\Rmods \A$, the induced sequence 
\begin{displaymath}
0\to Y\otimes_\A L\to Y\otimes_\A M\to Y\otimes_\A N\to 0
\end{displaymath}
is exact. (This condition need only be checked when $Y$ is finitely presented.)
\item For any pure-injective $Y\in\Rmods \A$, the induced sequence 
\begin{displaymath}
0\to Y\otimes_\A L\to Y\otimes_\A M\to Y\otimes_\A N\to 0
\end{displaymath}
is exact.
\item The induced sequence
\begin{displaymath}
0\to N^*\to M^*\to L^*\to 0
\end{displaymath}
is split exact in $\Rmods \A$.
\end{enumerate}
\end{cor}
\begin{proof}
A functor $F\in(\A\lmods,\Ab)$ is:
\begin{itemize}
\item finitely presented if and only if it comes from a pp-pair \cite[Section 10.2.5]{prest2009}.
\item fp-injective if and only if it is of the form $Y\otimes_\A-$ for some $Y\in\Rmods \A$ by \cite[Theorem 12.1.6]{prest2009}.
\item injective if and only if it is of the form $Y\otimes_\A-$ for some pure-injective $Y\in \Rmods \A$ by \cite[Theorem 12.1.6]{prest2009} (uses the fact that 1 is equivalent to 4).
\end{itemize}
For each $X\in \A\Lmods$, there is an isomorphism $(-,X)|^*_{(\A\lmods)^\op}\cong X^*\otimes_\A-|_{\A\lmods}$ which is natural in $X$  \cite[3.2.11]{enochsjenda2011}. Therefore,  
\begin{displaymath}
0\to (-,N)|^*_{(\A\lmods)^\op}\to (-,M)|^*_{(\A\lmods)^\op}\to (-,L)|^*_{(\A\lmods)^\op}\to
\end{displaymath}
is split exact if and only if
\begin{displaymath}
0\to N^*\otimes_\A-|_{\A\lmods}\to M^*\otimes_\A-|_{\A\lmods}\to L^*\otimes_\A-|_{\A\lmods}\to 0
\end{displaymath}
is split exact. Since $\Rmods \A\to(\A\lmods,\Ab):Y\mapsto Y\otimes-|_{\A\lmods}$ is fully faithful, this is equivalent to 6.
\end{proof}

\bibliographystyle{amsplain}
\bibliography{cites}


\end{document}